\documentclass[12pt,reqno]{amsart}

\usepackage{verbatim}
\usepackage{amssymb}
\usepackage{xcolor}
\usepackage{graphicx}

\newcommand{\gam}{\gamma}
\newcommand{\eps}{\varepsilon}
\newcommand{\kap}{\kappa}
\newcommand{\sig}{\sigma}

\newcommand{\Gam}{\Gamma}
\newcommand{\Del}{\Delta}

\newcommand{\cH}{{\mathcal H}}

\newcommand{\bg}{{\mathbf g}}

\newcommand{\C}{{\mathbb C}}
\newcommand{\F}{{\mathbb F}}

\newcommand{\cS}{{\mathcal S}}

\newcommand{\hf}{{\hat f}}
\newcommand{\hh}{{\hat h}}

\newcommand{\ochi}{{\overline\chi}}
\newcommand{\opsi}{{\overline\psi}}

\newcommand{\hG}{{\widehat G}}
\newcommand{\hH}{{\widehat H}}
\newcommand{\hR}{{\widehat R}}

\newcommand{\wh}{\widehat}

\newcommand{\longc}{,\dotsc,}
\newcommand{\longp}{+\dotsb+}

\newcommand{\longu}{\cup\dotsb\cup}

\newcommand{\sbs}{\subset}
\newcommand{\seq}{\subseteq}
\newcommand{\stm}{\setminus}

\renewcommand{\cline}[4]{\color{#1}{\line(#2,#3){#4}}}

\newcommand{\sub}[1]{_{\substack{#1}}}

\DeclareMathOperator{\supp}{supp}
\DeclareMathOperator{\PG}{PG}

\theoremstyle{plain}
\newtheorem{lemma}{Lemma}
\newtheorem{theorem}{Theorem}
\newtheorem{corollary}{Corollary}
\newtheorem{conjecture}{Conjecture}

\newtheorem{alphatheorem}{Theorem}

\theoremstyle{remark}
\newtheorem{remark}{Remark}

\newcommand{\refb}[1]{\cite{b:#1}}
\newcommand{\refe}[1]{\eqref{e:#1}}
\newcommand{\refc}[1]{\ref{c:#1}}
\newcommand{\reff}[1]{\ref{f:#1}}
\newcommand{\refj}[1]{\ref{j:#1}}
\newcommand{\refl}[1]{\ref{l:#1}}
\newcommand{\refs}[1]{\ref{s:#1}}
\newcommand{\reft}[1]{\ref{t:#1}}

\title[Uncertainty in finite planes]%
  {Uncertainty in finite planes}

\author[Andr\'as Bir\'o]{Andr\'as Bir\'o$^\dag$}
\email{biro.andras@renyi.mta.hu}
\address{A. R\'enyi Institute of Mathematics, Hungarian Academy of Sciences,
  1053 Budapest, Re\'altanoda u. 13--15, Hungary}
\thanks{${}^\dag$ Research partially supported by NKFIH (National Research,
  Development and Innovation Office) grants K-109789, K~119528 and by the MTA
  R\'enyi Int\'ezet Lend\"ulet Automorphic Research Group.}

\author{Vsevolod F. Lev}
\email{seva@math.haifa.ac.il}
\address{Department of Mathematics, The University of Haifa at Oranim,
  Tivon 36006, Israel}


\begin{document}
\baselineskip=16pt

\begin{abstract}
We establish a number of uncertainty inequalities for the additive group of a
finite affine plane, showing that for $p$ prime, a nonzero function
$f\colon\F_p^2\to\C$ and its Fourier transform $\hf\colon\wh{\F_p^2}\to\C$
cannot have small supports simultaneously. The ``baseline'' of our
investigation is the well-known Meshulam's bound, which we sharpen, for the
particular groups under consideration, taking into account not only the sizes
of the support sets $\supp f$ and $\supp\hf$, but also their structure.

Our results imply in particular that, with some explicitly classified
exceptions, one has $|\supp f||\supp\hf|\ge3p(p-2)$; in comparison, the
classical uncertainty inequality gives $|\supp f||\supp\hf|\ge p^2$.
\end{abstract}

\maketitle

\section{Introduction and background}\label{s:intro}

The uncertainty principle asserts that a nonzero function and its Fourier
transform cannot be both highly concentrated on small sets. In this paper we
will be concerned with Fourier analysis on finite abelian groups with the
uniform probability measure, the most known and classical realization of the
general uncertainty principle in these settings being as follows (see, for
instance, \cite{b:ds,b:sm,b:te}).
\begin{alphatheorem}\label{t:basic}
If $G$ is a finite abelian group, then for any nonzero function $f\in L(G)$
one has
  $$ |\supp f||\supp\hf| \ge |G|. $$
\end{alphatheorem}

In the statement of Theorem~\reft{basic} and throughout, we denote by $L(G)$
the vector space of all complex-valued functions on the finite abelian group
$G$, and by $\hf$ the Fourier transform of a function $f\in L(G)$ with
respect to the unform probability measure; that is,
  $$ \hf(\chi)=\frac1{|G|}\,\sum_{g\in G} f(g)\,\ochi(g),\quad \chi\in\hG $$
where $\hG$ is the group dual to $G$, and $\ochi$ is the character conjugate
to $\chi$. (See Section~\refs{notation} for the summary of notation used.)

Theorem~\reft{basic} can be significantly improved for groups of prime order,
which we identify with the additive groups of the corresponding fields and
denote $\F_p$.
\begin{alphatheorem}[Bir\'o~\refb{bi}, Tao~\refb{ta}]\label{t:birotao}
If $p$ is a prime, then for any nonzero function $f\in L(\F_p)$ one has
  $$ |\supp f|+|\supp\hf| \ge p+1. $$
\end{alphatheorem}
The inequality of Theorem~\reft{birotao} was established by the first-named
author of the present paper, who has contributed it as a problem to the year
1998 Mikl\'os Schweitzer mathematical competition, and then independently
rediscovered by Tao. Tao has also shown that the inequality is sharp, and
provided some applications.

Theorem~\reft{birotao} has been extended by Meshulam onto arbitrary finite
abelian groups.
\begin{alphatheorem}[Meshulam~\refb{me}]\label{t:meshulam}
Suppose that $G$ is a finite abelian group, and $f\in L(G)$. If $d_1<d_2$ are
two consecutive divisors of $|G|$ such that $d_1\le|\supp f|\le d_2$, then
  $$ |\supp\hf| \ge \frac{|G|}{d_1d_2}\,(d_1+d_2-|\supp f|). $$
\end{alphatheorem}
Notice that in the case where $G=\F_p$ with $p$ prime,
Theorem~\reft{meshulam} reduces to Theorem~\reft{birotao}. Indeed, Meshulam's
proof of Theorem~\reft{meshulam} uses induction, with Theorem~\reft{birotao}
serving the base case.

As it has been observed by Tao, Theorem~\reft{meshulam} shows that in the
Euclidean plane, the points $(|\supp f|,|\supp\hf|)$ lie on or above the
convex polygonal line through the points $(|H|,|G/H|)$, where $H$ ranges over
all subgroups of $G$. At the same time, Theorem~\reft{basic} merely states
that the points $(|\supp f|,|\supp\hf|)$ lie on or above the hyperbola
determined by the points $(|H|,|G/H|)$.

Suppose that $H$ is a subgroup, and $g$ is an element of a finite abelian
group $G$. Let
 $$ H^\perp:=\{\chi\in\hG\colon H\le\ker\chi\}. $$
It is a basic fact that a function $f\in L(G)$ is a scaled restriction of a
character $\psi\in\hG$ onto the coset $g+H$ if and only if the Fourier
transform $\hf$ is a scaled restriction of the evaluation homomorphism
$\chi\mapsto\ochi(g)$ onto the coset $\psi H^\perp$. We have then $|\supp
f|=|H|$ and $|\supp\hf|=|H^\perp|=|G|/|H|$, so that Theorem~\reft{basic} is
sharp in this case. Tao conjectured, however, that the estimate of
Theorem~\reft{basic} can be substantially sharpened, provided that
 $|\supp f|$ and $|\supp\hf|$ stay away from any divisor of $|G|$.
Theorem~\reft{meshulam} confirms this conjecture.

\section{Summary of results}\label{s:sumres}

It is well-known that the construction at the end of the previous section is
the only one for which equality holds in Theorem~\reft{basic}. This makes it
plausible to expect that, in fact, it might be possible to improve the
estimate of Theorem~\reft{basic} assuming only that $\supp f$ is not ``too
close'' to a coset of a subgroup of $G$, and $\supp\hf$ is not ``too close''
to a coset of a subgroup of $\hG$ (in contrast with the much stronger
assumption that $|\supp f|$ and $|\supp\hf|$ stay away from any divisor of
$|G|$). In this paper we establish several results of this sort in the
special case where the underlying group is elementary abelian of rank $2$;
that is, $G=\F_p^2$ with $p$ prime.

For comparison purposes, we notice that for the rank-$2$ elementary abelian
groups, Theorem~\reft{meshulam} can be rendered in a rather different way.
Namely, suppose that $f\in L(\F_p^2)$ is a nonzero function.
If $\min\{|\supp f|,|\supp\hf|\}\ge p$, then
\begin{equation}\label{e:meshalt}
  \min\{|\supp f|,|\supp\hf|\}
                           + \frac1p\,\max\{|\supp f|,|\supp\hf|\} \ge p+1;
\end{equation}
otherwise $\max\{|\supp f|,|\supp\hf|\}\ge p$ by Theorem~\reft{basic}, and
then Theorem~\reft{meshulam} leads to
  $$ |\supp\hf| \ge \begin{cases}
       p(p+1-|\supp f|)\ &\text{if}\ |\supp f|\le p\le |\supp\hf|, \\
       p^{-1}(p^2+p-|\supp f|)\ &\text{if}\ |\supp\hf|\le p\le |\supp f|,
             \end{cases} $$
which shows that~\refe{meshalt} holds true in this case, too. It is equally
easy to see that, conversely, for the groups under consideration,
\refe{meshalt} implies the estimate of Theorem~\reft{meshulam}. Thus,
\refe{meshalt} is an equivalent restatement of Theorem~\reft{meshulam} for
the groups $G=\F_p^2$.

We conjecture that, perhaps, much more can be true.
\begin{conjecture}\label{j:conjecture}
If $p$ is a prime, then for any nonzero function $f\in L(\F_p^2)$, and any
integer $k\in[1,p]$, writing for brevity $S:=\supp f$ and $X:=\supp\hf$, we
have
  $$ \frac1k\,\min\{|S|,|X|\}
                   + \frac1{p+1-k}\,\max\{|S|,|X|\} \ge p+1, $$
unless at least one of the sets $S\seq\F_p^2$ and $X\seq\wh{\F_p^2}$ is a
dense subset of a union of a small number of proper cosets of the
corresponding group. (Perhaps, it suffices to assume that neither $S$, nor
$X$ can be covered by fewer than $\min\{k,p+1-k\}$ cosets.)
\end{conjecture}

Equivalently, if for some $k\in[1,p]$ and $\eps\in(0,1)$ we have
$\min\{|S|,|X|\}\le(1-\eps)k(p+1)$, then
$\max\{|S|,|X|\}\ge\eps(p+1)(p+1-k)$, unless the subgroup structure of
$\F_p^2$ is involved, as indicated.

We will occasionally use the notation $S=\supp f$ and $X=\supp\hf$ without
redefining it anew each time.

The left-hand side of the inequality of Conjecture~\refj{conjecture} is
minimized, over all \emph{real} $k>0$, for
  $$ k = \frac{p+1}{\max\big\{\sqrt{|X|/|S|},\sqrt{|S|/|X|}\big\}+1}
                                                       \le \frac{p+1}2, $$
the corresponding minimum value being $(\sqrt{|X|}+\sqrt{|S|})^2/(p+1)$. As a
result, recalling Theorem~\reft{birotao}, it is very tempting to further
conjecture, as an ``almost-corollary'' of Conjecture~\refj{conjecture}, that
for any nonzero function $f\in L(\F_p^2)$ one has
\begin{equation}\label{e:roots}
  \sqrt{|X|}+\sqrt{|S|} \ge p+1,
\end{equation}
provided that neither $S$ nor $X$ is contained in a union of fewer than $p/2$
cosets.

The bounds of Conjecture~\refj{conjecture} corresponding to various values of
$k$, along with the enveloping bound~\refe{roots}, are shown in
Figure~\reff{Conjecture}. The yellow dots are points of the form
$(m(p+1-n),n(p+1-m))$ where $1\le m,n\le p$ are integers; their relevance
will become clear later.


\definecolor{noblue}{rgb}{0.6,0.6,0}
\begin{figure}[h!]
\caption{Conjecture~\refj{conjecture}.}
\center{\includegraphics[width=0.5\textwidth,scale=0.25]{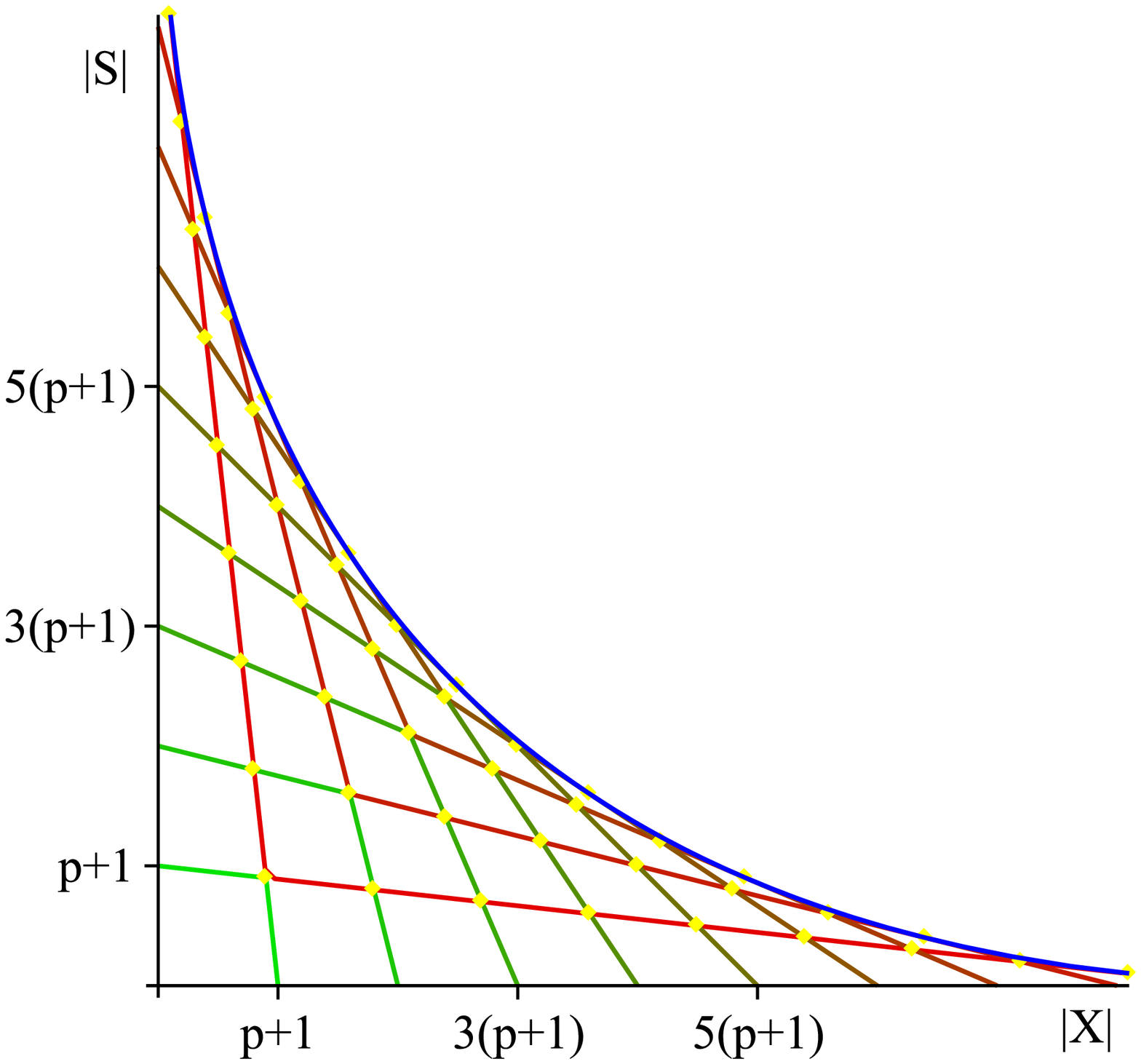}}

\begin{minipage}[c]{\textwidth}
\setlength{\unitlength}{1mm}
\begin{picture}(40,25)(-87,-94)
\begin{footnotesize}
\linethickness{0.4mm}
\put(0,0){\raisebox{0.3em}{\cline{red}{1}{0}{10}} $\ k=1$}
\put(0,-5){\raisebox{0.3em}{\cline{noblue}{1}{0}{10}} $\ k=(p+1)/2$}
\put(0,-10){\raisebox{0.3em}{\cline{green}{1}{0}{10}} $\ k=p$}
\put(0,-15){\raisebox{0.3em}{\cline{blue}{1}{0}{10}}
                                        $\ \sqrt{|S|}+\sqrt{|X|}\ge p+1$}
\end{footnotesize}
\end{picture}
\end{minipage}

\label{f:Conjecture} 
\end{figure}

\vskip -1.0in

The case $k=1$ of Conjecture~\refj{conjecture} is Theorem~\reft{meshulam} in
the form~\refe{meshalt}, the case $k=p$ follows from it since for any real
numbers $m\le M$, one has
  $$ \frac1p\,m+M\ge m+\frac1p\,M. $$
In general, for a positive integer $\kap<p/2$, the case $k=\kap$ of
Conjecture~\refj{conjecture} implies the case $k=p+1-\kap$, and for
$p/2<\kap<p$, the case $k=\kap$ implies the case $k=\kap+1$.

Our first principal result establishes the case $k=2$ of the conjecture for
\emph{rational-valued} functions.

\begin{theorem}\label{t:rational}
If $p\ge 3$ is a prime, and $f\in L(\F_p^2)$ is a nonzero rational-valued
function, then writing $S:=\supp f$ and $X:=\supp\hf$, we have
  $$ \frac12\,\min\{|S|,|X|\} + \frac1{p-1}\,\max\{|S|,|X|\} \ge p+1, $$
except if there exists a nonzero, proper subgroup $H<\F_p^2$ such that $f$ is
constant on each $H$-coset (in which case $X=H^\perp$ if the sum of all
values of $f$ is nonzero, and $X=H^\perp\stm\{1\}$ if the sum is equal to
$0$).
\end{theorem}

\begin{remark}
Denoting by $1_{H_1}$ and $1_{H_2}$ the indicator functions of distinct,
nonzero, proper subgroups $H_1,H_2<\F_p^2$, and letting $f:=1_{H_1}-1_{H_2}$,
we have $|S|=|X|=2(p-1)$, so that the estimate of Theorem~\reft{rational}
holds as an equality in this case.
\end{remark}

The reader is invited to review Figure~\reff{BoundComparison} where the
bounds of Theorems~\reft{basic} and~\reft{meshulam} are shown in gray (the
lower hyperbola) and black, respectively, and the bound of
Theorem~\reft{rational} is represented by the red dashed line.


\begin{figure}[h!]
\caption{Bound comparison.}
\center{\includegraphics[width=0.5\textwidth,scale=0.25]{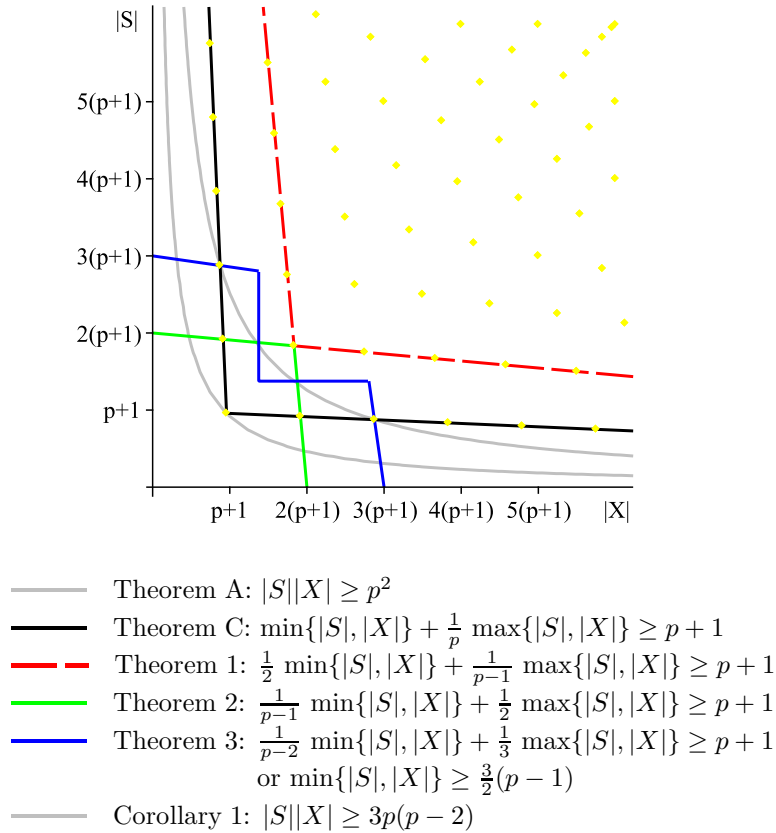}}

\begin{minipage}[c]{\textwidth}
\setlength{\unitlength}{1mm}
\begin{picture}(120,45)(-35,-40)
\begin{footnotesize}
\linethickness{0.4mm}
\put(0,0){\raisebox{0.3em}{\cline{lightgray}{1}{0}{10}}
  $\ $ Theorem~\reft{basic}: $|S||X|\ge p^2$}
\put(0,-5){\raisebox{0.3em}{\cline{black}{1}{0}{10}}
  $\ $ Theorem~\reft{meshulam}:
                          $\min\{|S|,|X|\}+\frac1p\,\max\{|S|,|X|\}\ge p+1$}
\put(0,-10){\raisebox{0.3em}{\cline{red}{1}{0}{6}}
\raisebox{0.3em}{\cline{red}{1}{0}{3}}
  $\ $ Theorem~\reft{rational}:
             $\frac12\,\min\{|S|,|X|\}+\frac1{p-1}\,\max\{|S|,|X|\}\ge p+1$}
\put(0,-15){\raisebox{0.3em}{\cline{green}{1}{0}{10}}
  $\ $ Theorem~\reft{k=p-1}:
             $\frac1{p-1}\,\min\{|S|,|X|\}+\frac12\,\max\{|S|,|X|\}\ge p+1$}
\put(0,-20){\raisebox{0.3em}{\cline{blue}{1}{0}{10}}
  $\ $ Theorem~\reft{k=p-2}:
          $\frac1{p-2}\,\min\{|S|,|X|\} + \frac13\,\max\{|S|,|X|\} \ge p+1$}
\put(0,-25){\phantom{\line(1,0){10} $\ $ Theorem 3: }
  or $\min\{|S|,|X|\}\ge\frac32(p-1)$}
  \put(0,-30){\raisebox{0.3em}{\cline{lightgray}{1}{0}{10}}
  $\ $ Corollary~\refc{uppergray}: $|S||X|\ge 3p(p-2)$}
\end{footnotesize}
\end{picture}
\end{minipage}
\vskip -0.5in

\label{f:BoundComparison} 
\end{figure}


Next, we settle the case $k=p-1$ of Conjecture~\refj{conjecture}. To state
the corresponding result, we notice that the definition of an orthogonal
subgroup at the end of Section~\refs{intro} establishes a bijection between
the subgroups of the group $\F_p^2$ and those of the dual group
$\wh{\F_p^2}$, the inverse bijection being given by
  $$ F \mapsto F^\perp
                 := \cap_{\chi\in F} \ker\chi,\quad F\le\wh{\F_p^2}. $$
We say that the subgroups $H\le\F_p^2$ and $H^\perp\le\wh{\F_p^2}$
(equivalently, $F\le\wh{\F_p^2}$ and $F^\perp\le\F_p^2$) are orthogonal to
each other.

\begin{theorem}\label{t:k=p-1}
If $p\ge 3$ is a prime, and $f\in L(\F_p^2)$ is a nonzero function, then
letting $S:=\supp f$ and $X:=\supp\hf$ we have
  $$ \frac1{p-1}\min\{|S|,|X|\} + \frac12\,\max\{|S|,|X|\} \ge p+1, $$
except if $S$ and $X$ are cosets of a pair of nonzero, proper, mutually
orthogonal subgroups of $\F_p^2$ and $\wh{\F_p^2}$, respectively.
\end{theorem}

The bound furnished by Theorem~\reft{k=p-1} is shown in green in
Figure~\reff{BoundComparison}.

\begin{remark} The inequality of the theorem readily implies
$\max\{|S|,|X|\}\ge 2(p-1)$.
\end{remark}

\begin{remark}
Equality is attained, for instance, if $f$ takes the value $1$ on a coset of
a nonzero, proper subgroup of $\F_p^2$, the value $-1$ on another coset of
the same subgroup, and vanishes outside of these two cosets; in this case
$|S|=2p$ and $|X|=p-1$.
\end{remark}

\begin{remark}
The exceptional case of the theorem is described by Lemma~\refl{cosets2} in
the appendix: namely, in this case there exist a nonzero, proper subgroup
$H<\F_p^2$, a character $\chi_0\in\wh{\F_p^2}$, an element $g_0\in\F_p^2$,
and a nonzero coefficient $c\in\C$ such that
  $$ f(g) = \begin{cases}
              c\chi_0(g)\ &\text{if}\ g\in g_0+H, \\
              0         \ &\text{if}\ g\notin g_0+H.
            \end{cases} $$
\end{remark}

Although we were unable to fully prove Conjecture~\refj{conjecture} for
$k=p-2$, we could at least give a proof under the extra assumption
$\min\{|S|,|X|\}<\frac32(p-1)$. The resulting estimate, visualized in
Figure~\reff{BoundComparison} by the blue line, can also be viewed as a
contribution towards the case $k=2$.

\begin{theorem}\label{t:k=p-2}
If $p\ge 3$ is a prime, and $f\in L(\F_p^2)$ is a nonzero function, then
letting $S:=\supp f$ and $X:=\supp\hf$, we have either
  $$ \frac1{p-2}\,\min\{|S|,|X|\} + \frac13\,\max\{|S|,|X|\} \ge p+1, $$
or
  $$ \min\{|S|,|X|\}\ge \frac32(p-1), $$
except if the smallest of the sets $S$ and $X$ is a coset of a nonzero,
proper subgroup of the corresponding group, possibly with one element
missing, and the largest is either a coset, or a union of two cosets of the
orthogonal subgroup.
\end{theorem}

\setcounter{remark}{0}
\begin{remark}
As an easy corollary of the theorem, if $p\ge 11$ is a prime, and $f\in
L(\F_p^2)$ is a nonzero function, then letting $S:=\supp f$ and
$X:=\supp\hf$, we have either
  $$ \min\{|S|,|X|\} \ge \frac32(p-1), $$
or
  $$ \max\{|S|,|X|\} \ge 3p-1, $$
unless $S$ and $X$ are exceptional as specified in the statement of the
theorem.
\end{remark}

\begin{remark}
The exceptional cases of the theorem are classified by Lemma~\refl{cosets2};
specifically, in these cases one of the following holds:
\begin{itemize}
\item[i)] there exist a nonzero, proper subgroup $H<\F_p^2$, an element
    $g_0\in\F_p^2$, characters $\chi_1,\chi_2\in\wh{\F_p^2}$ with
    $\chi_2\notin\chi_1H^\perp$, and coefficients $c_1,c_2\in\C$ at most
    one of which is equal to $0$, such that
    $$ f(g) = \begin{cases}
                c_1\chi_1(g)+c_2\chi_2(g)\ &\text{if}\ g\in g_0+H, \\
                0                        \ &\text{if}\ g\notin g_0+H;
              \end{cases} $$
\item[ii)] there exist a nonzero, proper subgroup $H<\F_p^2$, elements
    $g_1,g_2\in\F_p^2$ with $g_2\notin g_1+H$, a character
    $\chi_0\in\wh{\F_p^2}$, and coefficients $c_1,c_2\in\C$ at most one
    of which is equal to $0$, such that
    $$ f(g) = \begin{cases}
                c_i\chi_0(g)\ &\text{if}\ g\in g_i+H,\ i\in\{1,2\}, \\
                0           \ &\text{if}\ g\notin(g_1+H)\cup(g_2+H).
              \end{cases} $$
\end{itemize}
\end{remark}

As a consequence of Theorems~\reft{k=p-1}, \reft{k=p-2}, and~\reft{meshulam},
we have
\begin{corollary}\label{c:uppergray}
If $p>3$ is a prime, and $f\in L(\F_p^2)$ is a nonzero function, then letting
$S:=\supp f$ and $X:=\supp\hf$ we have
  $$ |S||X| \ge 3p(p-2), $$
unless either $\min\{|S|,|X|\}\le 2$, or the smallest of the sets $S$ and $X$
is a coset of a nonzero subgroup of the corresponding group, possibly with
one element missing, and the largest is either a coset, or a union of two
cosets of the orthogonal subgroup.
\end{corollary}

For the reader not convinced by Figure~\reff{BoundComparison}, where the
estimate of Corollary~\refc{uppergray} corresponds to the upper gray
hyperbola, we include the formal proof in the appendix.

The proof of Corollary~\refc{uppergray} relies on
Theorems~\reft{k=p-1},~\reft{k=p-2}, and~\reft{meshulam}. Using, instead of
Theorem~\reft{k=p-2}, one of the Theorems~\reft{as2} and~\reft{as3} below,
one can find constants $K>3$ and $N$ such that $|S||X|\ge Kp^2$, unless $S$
or $X$ is contained in a union of $N$ proper cosets. Indeed, it is easy to
see that, assuming Conjecture~\refj{conjecture}, for \emph{any} real $K$
there exists $N=N(K)$ with the property just mentioned. While we seem to be
far from establishing Conjecture~\refj{conjecture} in full generality, we
feel that it may be possible to prove at least the estimate $|S||X|\ge Kp^2$
developing further the ideas behind the proofs of Theorems~\reft{as2}
and~\reft{as3}.

\begin{theorem}\label{t:as2}
If $p\ge 31$ is a prime and $f\in L(\F_p^2)$ is a nonzero function, then
writing $S:=\supp f$ and $X:=\supp\hf$, for any $\eps\in(0,1)$ we have either
  $$ \min\{|S|,|X|\} \ge 2(1-\eps)p $$
or
  $$ \max\{|S|,|X|\} \ge \eps p^{3/2}, $$
except if the smallest of the sets $S$ and $X$ is contained in a coset of a
nonzero, proper subgroup of the corresponding group.
\end{theorem}

\setcounter{remark}{0}
\begin{remark}
The bound $p\ge 31$ is certainly not best possible. It can be relaxed by
fine-tuning the parameters of our proof and, perhaps, can be dropped
altogether.
\end{remark}

\begin{remark}
The exceptional case of Theorem~\reft{as2}, where the smallest of the sets
$S$ and $X$ is contained in a coset of a proper subgroup, is directly
addressed in Lemma~\refl{cosets2}.
\end{remark}

\begin{remark}
To put Theorem~\reft{as2} in a context, the reader is recommended to review
the paragraph following Conjecture~\refj{conjecture} observing, on the other
hand, that the assertion of Theorem~\reft{as2} can be equivalently written as
  $$ \frac12\,\min\{|S|,|X|\} + \frac1{\sqrt p}\, \max\{|S|,|X|\} \ge p, $$
apart from the exceptional case specified in the theorem. Thus,
Theorem~\reft{rational} gives a stronger estimate than Theorem~\reft{as2},
while the latter theorem does not impose the rationality assumption.
\end{remark}

\begin{remark}
The coefficient $2(1-\eps)$ in the statement of Theorem~\reft{as2} cannot be
replaced with~$2$. This is readily seen by fixing two distinct nonzero,
proper subgroups $H_1,H_2<\F_p^2$, and letting $f$ to be the difference of
their indicator functions: $f=1_{H_1}-1_{H_2}$; in this case
$|S|=|X|=2(p-1)$.
\end{remark}

\begin{theorem}\label{t:as3}
If $p$ is a prime, and $f\in L(\F_p^2)$ is a nonzero function, then writing
$S:=\supp f$ and $X:=\supp\hf$, for any $\eps\in(0,1)$ we have either
  $$ \min\{|S|,|X|\} \ge 3\,(1-\eps)p, $$
or
  $$ \max\{|S|,|X|\} \ge \frac16\,\eps p^{4/3}, $$
except if the smallest of the sets $S$ and $X$ is contained in a coset of a
nonzero, proper subgroup, or in a union of two such cosets (possibly
corresponding to different subgroups).
\end{theorem}

\setcounter{remark}{0}
\begin{remark}
In the situation where $|X|\le|S|$, the exceptional cases of
Theorem~\reft{as3} are classified by
Lemmas~\refl{cosets2}--\refl{twononparallel} in the appendix; the situation
where $|S|\le|X|$ can be dealt with using duality.
\end{remark}

\begin{remark}
The two inequalities of Theorem~\reft{as3} can be merged together to read
  $$ \frac13 \min\{|S|,|X|\} + \frac{6}{p^{1/3}}\, \max\{|S|,|X|\} \ge p, $$
to be compared against the case $k=3$ of Conjecture~\refj{conjecture}.
\end{remark}

\begin{remark}
The coefficient $3(1-\eps)$ cannot be replaced with $3$. This is readily seen
by taking three pairwise distinct, nonzero, proper subgroups
$H_1,H_2,H_3<\F_p^2$, and letting $f:=1_{H_1}+1_{H_2}-2\cdot 1_{H_3}$; in
this case $|S|=|X|=3(p-1)$.
\end{remark}

Theorem~\reft{as3} is easily seen to imply Theorems~\reft{k=p-1}
and~\reft{k=p-2} for sufficiently large primes $p$, apart from the slightly
less accurate classification of the exceptional cases. We believe, hoverer,
that the two latter theorems are worth stating separately as their proofs are
short, non-technical, and based on the ideas distinct from those used in the
proof of Theorem~\reft{as3}.

In the next section we briefly summarize the basic definitions, notation, and
facts about the Fourier transform in finite abelian groups.
Section~\refs{observations} contains some simple, but important observations
preparing the ground for the proofs of Theorems~\reft{rational}--\reft{as3};
the proofs themselves are presented in
Sections~\refs{p-rational}--\refs{p-as3}, respectively. In the appendix we
state and prove Lemmas~\refl{cosets1}--\refl{twononparallel} classifying the
exceptional cases arising in Theorems~\reft{k=p-1}--\reft{as3}, and also
prove Corollary~\refc{uppergray}; these results were referred to above and,
with the exception of Lemma~\refl{cosets1}, are not used elsewhere.

\section{Fourier transform: notation and basics}\label{s:notation}

Although familiarity with Fourier transform is assumed, the brief review
below can be useful. For the reader's convenience, we include here the
notation that has already been introduced above. The proofs, on the other
hand, are omitted; they can be found in any standard textbook on the subject,
like~\refb{te}.

For a finite abelian group $G$, we denote by $L(G)$ the vector space of all
complex-valued functions on $G$, and by $\hG$ the dual character group. Every
finite abelian group is isomorphic to its dual, and is naturally isomorphic
to its ``double-dual''; this allows one to switch the roles of $G$ and $\hG$.

We are primarily interested in the situation where $G$ is the elementary
abelian $p$-group of rank $2$, which we denote $\F_p^2$, where $p$ is a
prime.

For a character $\chi\in\hG$, by $\ochi$ we denote the conjugate character;
that is, $\ochi(g)=\chi(-g)$ is the complex conjugate of $\chi(g)$, for any
$g\in G$. The principal character will be denoted $1$; thus, $1=1_G$, with
the convention that $1_A$ denotes the indicator function of the set $A$.

The Fourier transform of a function $f\in L(G)$ is the function
 $\hf\in L(\hG)$ defined by
  $$ \hf(\chi)
           = \frac1{|G|}\,\sum_{g\in G} f(g)\,\ochi(g),\quad \chi\in\hG, $$
and the inversion formula is
  $$ f(g) = \sum_{\chi\in\hG} \hf(\chi)\,\chi(g), \quad g\in G. $$
The values $\hf(\chi)$ are called the Fourier coefficients of the function
$f$.

The convolution $f_1\ast f_2$ of the functions $f_1,f_2\in L(G)$ is defined
by
  $$ f_1\ast f_2\colon g
       \mapsto \frac1{|G|}\,\sum_{\sub{g_1,g_2\in G\\ g_1+g_2=g}}
                                          f_1(g_1)f_2(g_2),\quad g\in G. $$
We have $\wh{f_1\ast f_2}=\wh{f_1}\cdot\wh{f_2}$ and, conversely,
$\wh{f_1f_2}=\wh{f_1}\ast\wh{f_2}$ for any $f_1,f_2\in L(G)$, with the
convolution on the dual group defined by
  $$ u_1\ast u_2\colon \chi \mapsto
            \sum_{\sub{\chi_1,\chi_2\in\hG \\ \chi_1\chi_2=\chi}}
            u_1(\chi_1)u_2(\chi_2), \quad \chi\in\hG, $$
where $u_1,u_2\in L(\hG)$. (The minor normalization inconsistency arising
here can be formally resolved by looking at the ordered pairs $(G,\hG)$
instead of single groups $G$.)

The subgroup of $\hG$ orthogonal to a given subgroup $H\le G$ is
  $$ H^\perp := \{\chi\in\hG\colon H\le\ker\chi \}, $$
and the subgroup of $G$ orthogonal to a given subgroup $F\le\hG$ is
  $$ F^\perp := \cap_{\chi\in F} \ker\chi. $$
The subgroup $H^\perp$ is naturally isomorphic to the character group
$\wh{G/H}$.

We have $\wh{1_H}=(|H|/|G|)\cdot1_{H^\perp}$ and, more generally,
$\wh{1_{g+H}}(\chi)=(|H|/|G|)\,\ochi(g)\cdot1_{H^\perp}(\chi)$ for any
element
 $g\in G$ and character $\chi\in\hG$.

Finally, $(H^\perp)^\perp=H$ for any subgroup $H\le G$, and similarly
$(F^\perp)^\perp=F$ for any subgroup $F\le\hG$; as a result, one can speak
about pairs of mutually orthogonal subgroups.

\section{Basic observations}\label{s:observations}

Let $G$ be a finite abelian group.

For a function $f\in L(G)$, a subgroup $H\le G$, and an element $g\in G$, the
Fourier coefficients of the function $f\cdot 1_{g+H}$ (coinciding with $f$ on
the coset $g+H$ and vanishing outside of it) are
\begin{align}
  \wh{f\cdot1_{g+H}}(\chi)
    &= (\hf\ast\wh{1_{g+H}})(\chi) \notag \\
    &= \sum_{\psi\in\hG} \hf(\chi\psi) \cdot
                   \frac{\psi(g)}{|H^\perp|}\,1_{H^\perp}(\opsi) \notag \\
    &=\frac1{|H^\perp|}\sum_{\psi\in H^\perp}\hf(\chi\psi)\psi(g),
                                            \quad \chi\in\hG. \label{e:psf}
\end{align}
A more explicit form of this relation is
\begin{equation}\label{e:psiE}
  \sum_{\psi\in H^\perp}\hf(\chi\psi)\psi(g)
         = \frac{\ochi(g)}{|H|}\,\sum_{h\in H} f(g+h)\,\ochi(h),
                                                          \quad \chi\in\hG.
\end{equation}

Given a subgroup $H\le G$ and a nonzero function $f\in L(G)$, let
 $S:=\supp f$ and $X:=\supp\hf$, and denote by $n_S$ the smallest positive
number of elements of $S$ contained in a coset of $H$, and by $n_X$ the
smallest positive number of characters from $X$ contained in a coset of
$H^\perp$:
\begin{equation}\label{e:nSnX}
  n_S := \min\{|(s+H)\cap S|\colon s\in S\},\ %
                      n_X:=\min\{|\chi H^\perp\cap X|\colon \chi\in X \}.
\end{equation}
Also, let $K_S$ be the number of $H$-cosets having a nonempty intersection
with $S$, and let $K_X$ be the number of $H^\perp$-cosets having a nonempty
intersection with $X$:
\begin{equation}\label{e:KSKX}
   K_S := |S+H|/|H|,\ K_X := |XH^\perp|/|H^\perp|.
\end{equation}
Thus, $S$ and $X$ depend on $f$, while $n_S,n_X,K_S$, and $K_X$ depend on
both $f$ and $H$, although this dependence is not reflected explicitly by our
notation.

Recall that a group is called \emph{prime} if it has prime order (in which
case it is cyclic).
\begin{lemma}\label{l:SXmn} 
Suppose that $H$ is a subgroup of the finite abelian group $G$, and
 $f\in L(G)$ is a nonzero function, and let $S,X,n_S,n_X,K_S$, and $K_X$ be
as above. If $H$ is prime, then $K_X\ge|H|+1-n_S$, whence
 $|X|\ge n_X(|H|+1-n_S)$. Similarly, if $H$ is co-prime (meaning that
$G/H$ is prime), then $K_S\ge|H^\perp|+1-n_X$, whence
 $|S|\ge n_S(|H^\perp|+1-n_X)$.
\end{lemma}

\begin{proof}
Fix $g\in G$ with $|(g+H)\cap S|=n_S$, and consider the function
 $f_g\in L(H)$ defined by $f_g(h):=f(g+h),\ h\in H$. In terms of this
function, \refe{psiE} can be rewritten as
\begin{equation}\label{e:psiEprime}
  \ochi(g)\,\wh{f_g}(\chi\vert_H)
       = \sum_{\psi\in H^\perp} \hf(\chi\psi)\psi(g), \quad \chi\in\hG,
\end{equation}
where $\chi\vert_H$ denotes the restriction of $\chi$ onto $H$. Since
 $|\supp f_g|=n_S$ and $H$ is prime, by Theorem~\reft{birotao} there are at
least $|H|+1-n_S$ characters $\eta\in\hH$ with $\wh{f_g}(\eta)\ne 0$. Every
such character $\eta\in\hH$ extends to a character $\chi\in\hG$ with
$\chi\vert_H=\eta$. For this character $\chi$, the left-hand side
of~\refe{psiEprime} is nonzero; hence, the right-hand side is nonzero either,
showing that $\chi H^\perp$ has a nonempty intersection with $X$. Moreover,
if $\chi'\vert_H\ne\chi''\vert_H$, then $\chi'H^\perp\ne\chi'' H^\perp$, so
that different characters
 $\eta\in\hH$ result in different cosets $\chi H^\perp$.

This proves the first assertion of the lemma. The second one follows by
duality; that is, essentially, by repeating the argument with $G$, $H$, and
$f$ replaced with $\hG$, $H^\perp$, and $\hf$, respectively (which is
legitimate since $|H^\perp|=|\wh{G/H}|=|G/H|$ shows that $H^\perp$ is prime).
\end{proof}

Although fairly straightforward, Lemma~\refl{SXmn} is of crucial importance
for the proofs of Theorems~\reft{rational}--\reft{as3}.

It may be worth noting that the argument employed in the proof of
Lemma~\refl{SXmn} can be used to give an inductive proof of
Theorem~\reft{basic}. Namely, choosing arbitrarily a nonzero proper subgroup
$H<G$ (the induction basis where $G$ is a prime group is to be given a
separate treatment), and using the induction hypothesis instead of the
assumption that $H$ and $G/H$ are prime, we get $|X|\ge n_X\cdot|H|/n_S$ and
$|S|\ge n_S\cdot|H^\perp|/n_X$, which yields $|S||X|\ge|H||H^\perp|=|G|$.

As another illustration of our approach, we derive Theorem~\reft{meshulam}
for the group $G=\F_p^2$. By duality, we can assume that $|X|\le|S|$. Fix a
nonzero, proper subgroup $H<\F_p^2$, and define $n_S,n_X,K_S,K_X$ as above.
By Lemma~\refl{SXmn},
\begin{align*}
  |X|+\frac1p\,|S|
    &\ge n_X(p+1-n_S) + \frac1p\,n_S(p+1-n_X) \\
    &= p+1 + \frac{p+1}p\,(n_X-1)(p-n_S) \\
    &\ge p+1,
\end{align*}
as wanted.

\section{Proof of Theorem~\reft{rational}}\label{s:p-rational}

Recall, that for a prime power $q$, a blocking set in the affine plane
$\F_q^2$ is a set that blocks (meets) every line. A union of two nonparallel
lines is a blocking set of size $2q-1$, and a classical result by
Jamison~\refb{ja} and Brouwer-Schrijver~\refb{bs} (see also
\cite{b:al,b:bcps}) says that, in fact, any blocking set in $\F_q^2$ has size
at least $2q-1$. We need a stability version of this result.
\begin{lemma}\label{l:pencils}
Suppose that $q$ is a prime power, $k$ and $m$ are positive integers, and
$S\seq\F_q^2$ is a set blocking every line in $\F_q^2$ with the exception of
at most $k$ pencils of parallel lines, each of these pencils containing at
most $m$ nonblocked lines. Then $|S|\ge2q-k-m$.
\end{lemma}

\begin{proof}
We refer the directions of the nonblocked lines as \emph{special}; thus,
there are at most $k$ special directions.

If $S$ is a blocking set, then $|S|\ge 2q-1$ and the proof is over. Suppose
thus that there is a line $l\seq\F_q^2$ avoiding $S$. Notice that the
direction of $l$ is special.

Consider an embedding of $\F_q^2$ into the projective plane $\PG(2,q)$, with
$\PG(2,q)\stm\F_q^2$ designated as the \emph{line at infinity}. Let
$\cS\seq\PG(2,q)$ be the set consisting of (the image of) $S$ and the points
at infinity corresponding to the special directions; thus, $|\cS|\le|S|+k$.
Also, let $\ell$ be the line in $\PG(2,q)$ containing $l$, and let $\wp$ be
the point of infinity incident with $\ell$; that is, $\{\wp\}=\ell\stm l$.

Clearly, $\cS$ blocks every line in $\PG(2,q)$, with $\ell$ being blocked by
the point $\wp$ only. Consequently, the set $\cS\stm\{\wp\}$ blocks every
line in $\PG(2,q)$, excepting $m$ lines at most.

We now get back to the affine world by identifying $\PG(2,q)\stm\ell$ with
$\F_q^2$. Corresponding to the set $\cS\stm\{\wp\}$ under this identification
is a set $S'\seq\F_q^2$ which blocks every line in $\F_q^2$ with the possible
exception of at most $m$ lines. Adding at most $m$ points to this set, we get
a blocking set in $\F_q^2$, whence $|S'|\ge (2q-1)-m$ by the
Jamison-Brouwer-Schrijver result. It follows that
  $$ |S| \ge |\cS|-k = |S'|+1-k \ge (2q-1)-m + 1 - k = 2q-k-m, $$
as wanted.
\end{proof}

Turning to the proof of Theorem~\reft{rational}, we write for brevity
$G:=\F_p^2$, and identify $G$ and $\hG$ with the additive group of the
two-dimensional vector space over the field $\F_p$; thus, we call the
elements of $G$ and $\hG$ \emph{points}, and cosets of their nonzero, proper
subgroups \emph{lines}. If $H<G$ is a nonzero, proper subgroup, then
$H$-cosets in $G$ will be referred to as $H$-lines, and $H^\perp$-cosets in
$\hG$ as $H^\perp$-lines. Notice that the origin of $\hG$ is the principal
character.

For a character $\chi\in\hG$, we have $\chi\in X$ if and only if the sum
$\sum_{g\in S} f(g)\ochi(g)$ is a nonzero element of the cyclotomic field of
order $p$. As an immediate corollary, if $\chi\in X$, then also $\chi^j\in X$
for each $j\in[1,p-1]$; that is, $X$ is a union of several proper subgroups
of $\hG$, with the possible exception of the principal character that can be
missing from $X$. In other words, $X\cup\{1\}$ is a union of lines in $\hG$
passing through the origin. It follows that either $X\cup\{1\}$ is a proper
subgroup of $\hG$, in which case the assertion is immediate from
Lemma~\refl{cosets1} ii), or $|X|\ge 2(p-1)$, which readily gives the
estimate sought in the case where $|S|\ge|X|$.

Suppose therefore that $|S|<|X|$ and then, for a contradiction, that
\begin{equation}\label{e:rat-contra}
  \frac12\,|S| + \frac1{p-1}\,|X| < p+1.
\end{equation}
Fix a nonzero, proper subgroup $H<G$, and let the quantities
$n_S,n_X,K_S,K_X$ be defined by~\refe{nSnX} and~\refe{KSKX}. Substituting the
inequalities
  $$ |X|\ge n_X(p+1-n_S),\ |S|\ge n_S(p+1-n_X) $$
of Lemma~\refl{SXmn} into~\refe{rat-contra}, after routine algebraic
manipulations we get $(n_S-2)(p-1-n_X)<0$; thus, either $n_X=p$, or $n_S=1$.
In the former case $X$ is a union of $H^\perp$-lines, and since $X$
intersects nontrivially every $H^\perp$-line not passing through the origin,
all such lines are in fact contained in $X$; hence
 $|X|\ge |\hG|-|H^\perp|=p^2-p$, implying $|S|=1$ in view
of~\refe{rat-contra} and readily leading to a contradiction.

We therefore have $n_S=1$, for any choice of the subgroup $H<G$. Applying
Lemma~\refl{SXmn}, we conclude that $K_X=p$, and it follows that $X$ contains
the principal character (other\-wise any line through the origin, not
contained in $X$, would have an empty intersection with $X$). Consequently,
$X$ is a union of nonzero, proper subgroups of $\hG$.

Denote by $\cH$ be the set of all those proper subgroups $H<G$ with
$H^\perp\seq X$, and write $k:=|\cH|$; thus, $|X|=k(p-1)+1$. If we had
$k=1$, then $X$ were a subgroup and Lemma~\refl{cosets1} would show that
$S$ is a union of $X$-cosets, contrary to our present assumption
$|S|<|X|$; thus, $k\ge 2$. Clearly, we have $n_X=k-1$ for every subgroup
$H\in\cH$, and $n_X=1$ for every subgroup $H\notin\cH$. As a result,
Lemma~\refl{SXmn} shows that $S$ meets every line in $G$, except that in
each of the $k$ directions corresponding to the subgroups $H\in\cH$,
there can be up to $k-2$ lines avoiding $S$. By Lemma~\refl{pencils}, we
have
  $$ |S|\ge 2p-k-(k-2). $$
Recalling that $|X|=k(p-1)+1$, we obtain
  $$ \frac12\,|S| + \frac1{p-1}\,|X| > (p-k+1) + k = p+1, $$
which proves the assertion.

\section{Proof of Theorem~\reft{k=p-1}}\label{s:p-k=p-1}

In this section and also in Sections~\refs{p-k=p-2}-\refs{p-as3} below we
keep using the conventions of the previous section, writing $G:=\F_p^2$ and
using geometric terminology for the elements and subgroups of $G$ and $\hG$.

By duality, we can assume that
\begin{equation}\label{e:XleSth7}
  |X| \le |S|,
\end{equation}
and then for a contradiction that
\begin{equation}\label{e:XSsmallth7}
   \frac12\,|S| + \frac1{p-1}\,|X| < p+1.
\end{equation}

Fix a nonzero, proper subgroup $H<G$, and let $n_S$ and $n_X$ be defined
by~\refe{nSnX}. By Lemma~\refl{SXmn},
\begin{align*}
  \frac12\,|S| + \frac1{p-1}\,|X|
     &\ge \frac12\,n_S(p+1-n_X) + \frac1{p-1}\,n_X(p+1-n_S) \\
     &=   p+1 + \frac{p+1}{2(p-1)}\,(p-n_X-1)(n_S-2).
\end{align*}
Comparing with~\refe{XSsmallth7}, we see that either $n_X=p$, or $n_S=1$.

If, for a subgroup $H<G$, we have $n_X=p$, then $X$ is a union of
$H^\perp$-cosets. Moreover, if in this case we had $|S|\ge 2p$, this would
imply
  $$ \frac12\,|S| + \frac1{p-1}\,|X| > p + 1, $$
contradicting~\refe{XSsmallth7}. Thus, $|X|\le|S|<2p$ by~\refe{XleSth7},
showing that $X$ is in fact a unique $H^\perp$-coset and then, in view of
Lemma~\refl{cosets1} ii), that $S$ is an $H$-coset.

To complete the proof, we consider the situation where $n_S=1$ for every
nonzero, proper subgroup $H$. By Lemma~\refl{SXmn}, in this case every line
in $\hG$ contains a point from $X$; that is, $X$ is a blocking set in $\hG$.
Applying the result by Jamison-Brouwer-Schrijver mentioned at the beginning
of Section~\refs{p-rational}, we conclude that $|X|\ge 2p-1$. Hence,
by~\refe{XleSth7},
  $$ \frac12\,|S| + \frac1{p-1}\,|X|
                      \ge \Big(\frac12+\frac1{p-1}\Big) (2p-1) > p+1, $$
in a contradiction with~\refe{XSsmallth7}.

\section{Proof of Theorem~\reft{k=p-2}}\label{s:p-k=p-2}

We assume, without loss of generality, that $|X|\le|S|$, and that
\begin{equation}\label{e:Xsmall323}
  |X| < \frac32\,(p-1)
\end{equation}
and
\begin{equation}\label{e:Ssmall323}
  \frac13\,|S| + \frac1{p-2}\, |X| < p+1,
\end{equation}
aiming to show that $S$ and $X$ have the structure detailed in the statement
of the theorem.

Notice that from~\refe{Ssmall323} and Theorem~\reft{meshulam},
  $$ p+1 > \frac13\,|S| + \frac1{p-2}\,\left( p+1-\frac1p\,|S| \right), $$
implying
\begin{equation}\label{e:Sl3p}
  |S|<3p.
\end{equation}

Fix a nonzero, proper subgroup $H<G$, and define $n_X,n_S,K_X,K_S$
by~\refe{nSnX} and ~\refe{KSKX}. Substituting the inequalities
\begin{equation}\label{e:LSXmn}
  |X|\ge n_X(p+1-n_S),\quad |S|\ge n_S(p+1-n_X)
\end{equation}
of Lemma~\refl{SXmn} into~\refe{Ssmall323}, simplifying, and factoring, we
get
  $$ (n_S-3)(p-2-n_X) < 0; $$
consequently, we have either $n_S\in\{1,2\}$, or $n_X\in\{p-1,p\}$. In the
latter case \refe{Xsmall323} yields $|X|<2n_X$, whence $X$ is contained in an
$H^\perp$-coset, and indeed $n_X\ge p-1$ shows that $X$ misses at most one
element of this coset. Moreover, substituting $|X|=n_X$ into~\refe{LSXmn}
gives $n_S=p$; along with~\refe{Sl3p}, this shows that $S$ is either a coset,
or a union of two cosets of $H$.

It thus remains to consider the situation where $n_S\in\{1,2\}$, for any
choice of a nonzero, proper subgroup $H<G$. By Lemma~\refl{SXmn}, in this
case we have $K_X\ge p-1$, meaning that for every given direction in $\hG$,
there is at most one line in that direction free of points of $X$. Since
there are $p+1$ directions, and any two lines in different directions meet in
exactly one point, we can add to $X$ at most $(p+1)/2$ points to get a set
which meets every line; that is, a blocking set. Recalling that, by a result
of Jamison-Brouwer-Schrijver (see the beginning of
Section~\refs{p-rational}), any blocking set in $\F_p^2$ has size at least
$2p-1$, we obtain
  $$ |X| \ge (2p-1)-\frac12\,(p+1) = \frac32\,(p-1), $$
a contradiction.

\section{Proof of Theorem~\reft{as2}}\label{s:p-as2}

We need the following lemma.
\begin{lemma}\label{l:lines}
For any prime $p$, and any finite set $P\seq\F_p^2$ with $2\le |P|\le 4p$,
not contained in a single line, there is a direction determined by $P$ such
that every line in this direction contains fewer than
$\sqrt{|P|}+\max\{1,|P|/(2p)\}$ points of $P$.
\end{lemma}

\begin{proof}
Denote by $d$ the number of directions determined by $P$. Sz\H onyi~\refb{sz}
has shown that if $|P|\le p$, then $d\ge\frac{|P|+3}2$; on the other hand, if
$|P|>p$, then among the $p$ lines in $\F_p^2$ in every given direction, there
must be a line containing two or more points of $P$, showing that $d=p+1$.
Thus, $d>\min\{|P|/2,p\}$ in any case.

Suppose that in every direction determined by $P$, there is a line containing
at least $M$ points of $P$; we want to show that
$M<\sqrt{|P|}+\max\{1,|P|/(2p)\}$. Let $l_1\longc l_d$ be lines in different
directions with $|l_i\cap P|\ge M$, for each $i\in[1,d]$. We use a well-known
consequence of the Cauchy-Schwartz inequality asserting that for any system
of finite sets $A_1\longc A_d$, one has
\begin{equation}\label{e:CaSc}
  (|A_1|\longp|A_d|)^2 \le |A_1\longu A_d|\,\sum_{i,j=1}^d |A_i\cap A_j|.
\end{equation}
We let $A_i:=l_i\cap P\ (i\in[1,d])$ and observe that then
 $|A_i\cap A_j|\le 1$ whenever $i\ne j$, and that $|A_1\longu A_d|\le|P|$.
Writing $\sig:=|A_1|\longp|A_d|$, from~\refe{CaSc} we obtain
$\sig^2\le|P|(d^2-d+\sig)$. It follows that
  $$ \Big(\sig-\frac12\,|P|\Big)^2
                                \le |P|(d^2-d) + \frac14\,|P|^2 < |P|d^2, $$
whence $\sig<\frac12|P|+d\sqrt{|P|}$. On the other hand, we have $\sig\ge
dM$. This yields $M<\sqrt{|P|}+|P|/(2d)$, and to complete the proof, we
recall that $d>\min\{|P|/2,p\}$.
\end{proof}

\begin{proof}[Proof of Theorem~\reft{as2}]
We recall our convention to write $G=\F_p^2$ and to think of $G$ and $\hG$
geometrically, referring to their elements and nonzero, proper subgroups as
points and lines, respectively.

Without loss of generality, we assume that $3\le|X|\le|S|$, and then for a
contradiction that
\begin{equation}\label{e:bothsmall}
  |X| < 2(1-\eps)p \quad\text{and}\quad |S| < \eps p^{3/2},
\end{equation}
while $X$ is not contained in a coset of a proper subgroup.

We notice that if $\eps\ge\frac12+\frac1{4\sqrt p}$, then
Theorem~\reft{meshulam} along with~\refe{bothsmall} gives
  $$ p+1 \le |X|+\frac1p\,|S| < 2(1-\eps)p + \eps\sqrt p
       \le \Big(1-\frac1{2\sqrt p}\Big)\,p
               + \Big(\frac12+\frac1{4\sqrt p}\Big)\sqrt p = p + \frac14, $$
a contradiction; thus,
\begin{equation}\label{e:epssmall}
  \eps < \frac12 + \frac1{4\sqrt p}\,.
\end{equation}

By Lemma~\refl{lines}, there is a nonzero, proper subgroup $H<G$ such that
every $H^\perp$-line contains fewer than $\sqrt{|X|}+1$ points of $X$ while,
on the other hand, there is an $H^\perp$-line containing at least two points
of $X$. Throughout the proof, we consider this subgroup $H$ fixed, and define
$n_S,n_X,K_S,K_X$ by~\refe{nSnX} and \refe{KSKX}.

The assumption that every $H^\perp$-line contains fewer than $\sqrt{|X|}+1$
points of $X$, in view of~\refe{bothsmall}, implies $n_X<\sqrt{2p}+1$.
Therefore, by~\refe{bothsmall} and Lemma~\refl{SXmn},
  $$ \eps p^{3/2} > |S| \ge n_S(p+1-n_X) > \frac12\, n_Sp, $$
which yields
\begin{equation}\label{e:mlog}
  n_S < 2\eps\sqrt p.
\end{equation}
Applying~\refe{bothsmall} and Lemma~\refl{SXmn} once again,
  $$ 2(1-\eps)p > |X| \ge n_X(p+1-n_S) > (1-\eps)n_Xp. $$
This gives $n_X=1$. As a result, by Lemma~\refl{SXmn}, we have $K_S=p$,
meaning that every $H$-line has a nonempty intersection with $S$. Hence,
averaging and using~\refe{bothsmall},
\begin{equation}\label{e:msqrt}
  n_S \le K_S^{-1}|S| < \eps\sqrt p
\end{equation}
(cf.~\refe{mlog}).

We say that a character $\chi\in X$ is \emph{isolated} if its $H^\perp$-line
does not contain any other character from $X$; that is, if
 $(\chi H^\perp)\cap X=\{\chi\}$. Let $N$ denote the number of isolated
characters; in other words, $N$ is the number of $H^\perp$-lines containing
exactly one point of $X$. Since
  $$ K_X\ge p+1-n_S $$
by Lemma~\refl{SXmn}, we have
   $$ |X| \ge N+2(p+1-n_S-N); $$
consequently,
\begin{equation}\label{e:manyisolated}
  N > 2(p-n_S)-|X|.
\end{equation}

For an element $g\in G$, let $k_g$ be the number of points of $S$ on the
$H$-line passing through $g$:
  $$ k_g = |(S-g)\cap H|; $$
that is, $k_g=|\supp(f\cdot1_{g+H})|$. By~\refe{msqrt}, there exists $g_0\in
G$ with
\begin{equation}\label{e:lg0small}
  k_{g_0} < \eps\sqrt p.
\end{equation}
Considering $g_0$ fixed, for each $g\in G$ we define the function
 $\Del_g\in L(G)$ by
  $$ \Del_g := f\ast\big(f\cdot(1_{g+H}-1_{g_0+H})\big), $$
and let $T:=\supp\Del_g$ and $Y:=\supp\wh{\Del_g}$ (thus, $T$ and $Y$ depend
on $g$). We have
  $$ \wh{\Del_g} = \hf\cdot(\wh{f\cdot1_{g+H}} - \wh{f\cdot1_{g_0+H}}) $$
whence, by~\refe{psf},
  $$ \wh{\Del_g}(\chi) = p^{-1} \hf(\chi)
        \,\sum_{\psi\in H^\perp}\hf(\chi\psi)\big(\psi(g)-\psi(g_0)\big),
                                                       \quad \chi\in \hG. $$
It follows that $Y\seq X$, and that $\chi\in X\stm Y$ if and only if
\begin{equation}\label{e:sumpsi}
  \sum_{\psi\in H^\perp}\hf(\chi\psi)\big(\psi(g)-\psi(g_0)\big) = 0.
\end{equation}
In particular, the $N$ isolated characters are all contained in $X\stm Y$;
therefore, letting $K_Y:=|YH^\perp|/|H^\perp|$,  by~\refe{manyisolated} we
get
\begin{gather}
  K_Y \le p-N < |X| - p + 2n_S \label{e:KYsmall}
  \intertext{and}
  |Y| \le |X|-N < 2(|X|-p+n_S), \notag
\end{gather}
the latter estimate implying
\begin{equation}\label{e:Yupper}
  |Y| < \Big(2-\frac72\,\eps\Big)p
\end{equation}
in view of \refe{bothsmall} and \refe{msqrt}. On the other hand,
\begin{equation}\label{e:Tupper}
  |T| \le |\supp(f\ast(f\cdot1_{g+H}))|
            + |\supp(f\ast(f\cdot1_{g_0+H}))| \le |S|(k_{g}+k_{g_0}).
\end{equation}

We notice that $\Del_g$ vanishes identically if and only if~\refe{sumpsi}
holds true for all characters $\chi\in X$. Assuming that \refe{sumpsi} is
\emph{wrong} for some character $\chi\in X$ and element $g\in G$ with
\begin{equation}\label{e:lgsmall}
  k_g\le\frac32\,\eps\sqrt p,
\end{equation}
so that, in particular, $\Del_g$ does \emph{not} vanish identically, we will
now get a contradiction.

To this end, we first observe that substituting~\refe{bothsmall},
\refe{lg0small}, and \refe{lgsmall} into~\refe{Tupper} yields
\begin{equation}\label{e:Texplicit}
  |T|< \frac52\,\eps^2 p^2.
\end{equation}
If $Y$ were situated on a single $H^\perp$-line, then by the assumption that
every $H^\perp$-line contains fewer than $\sqrt{|X|}+1$ elements of $X$, and
in view of $Y\seq X$ and~\refe{bothsmall}, we would have $|Y|<\sqrt{2p}+1$,
and then Theorem~\reft{meshulam} and~\refe{Texplicit} would give
  $$ p+1 \le |Y| + \frac1p\,|T| < \sqrt{2p} + 1 + \frac52\,\eps^2p, $$
contradicting~\refe{epssmall}. Thus, $Y$ resides on at least two distinct
$H^\perp$-lines; that is, $K_Y\ge 2$.

Let
  $$ n_T:=\min\{|(t+H)\cap T|\colon t\in T\} \quad\text{and}\quad
                   \ n_Y:=\min\{|(\chi H^\perp)\cap Y|\colon\chi\in Y\}. $$
By~\refe{Yupper} we have
  $$ n_Y \le |Y|/2 < \Big(1-\frac74\,\eps\Big)p $$
and then, in view of~\refe{Texplicit} and Lemma~\refl{SXmn},
  $$ \frac52\,\eps^2 p^2 > |T| \ge n_T(p+1-n_Y) > \frac74\,\eps n_Tp. $$
It follows that
  $$ n_T < \frac32\eps p, $$
whence
  $$ K_Y\ge p+1-n_T > \Big(1-\frac32\,\eps\Big)p $$
by Lemma~\refl{SXmn}. Therefore, from~\refe{KYsmall}, \refe{bothsmall},
and~\refe{msqrt}, we get
  $$ \Big(1-\frac32\eps\Big)p < K_Y < (1-2\eps)p + 2\eps\sqrt p, $$
a contradiction.

We therefore conclude that, letting
  $$ A := \Big\{g\in G\colon k_g\le\frac32\,\eps\sqrt p\Big\}, $$
equality~\refe{sumpsi} holds true for all characters $\chi\in X$ and elements
$g\in A$.

Clearly, the set $A$ is a union of $H$-lines, and we denote by $K_A$ the
number of these lines. By~\refe{bothsmall}, and since $K_S=p$, we have
  $$ \eps\,p^{3/2} > |S| \ge \frac32\,\eps\sqrt{p}\,(p-K_A), $$
resulting in
\begin{equation}\label{e:Klarge}
  K_A > \frac13\,p.
\end{equation}

Changing the viewpoint, we now fix a character $\chi\in X$ and denote the sum
in the left-hand side of~\refe{sumpsi} by $R(g)$, considering it as a
function of $g$. Observing that the term corresponding to the principal
character $\psi=1$ vanishes and can be dropped from the sum, we see that
$|\supp\hR|\le|\chi H^\perp\cap X|$, and that if $\chi$ is not isolated, then
$R$ does not vanish identically. On the other hand, we have shown that
$R(g)=0$ whenever $g\in A$, and it follows that
 $|\supp R|\le |G\stm A|=(p-K_A)p$. Using~\refe{meshalt}, we conclude that
for any nonisolated character $\chi\in X$,
  $$ |\chi H^\perp\cap X| + (p-K_A)
                             \ge |\supp\hR| + \frac1p\,|\supp R| \ge p+1, $$
implying
  $$ |\chi H^\perp\cap X| > K_A. $$
Recalling~\refe{Klarge} and~\refe{bothsmall}, this shows that any
$H^\perp$-line determined by $X$ contains, in fact, at least
$K_A>\frac13\,p>\sqrt{2p}+1>\sqrt{|X|}+1$ points of $X$. This, however,
contradicts the choice of $H$ at the beginning of the proof.
\end{proof}

\section{Proof of Theorem~\reft{as3}}\label{s:p-as3}

The proof of Theorem~\reft{as3} is a further elaboration on that of
Theorem~\reft{as2}.

In addition to Lemma~\refl{lines}, we need two more lemmas.

\begin{lemma}\label{l:richlines}
Suppose that $p\ge 3$ is a prime, and that a set $P\sbs\F_p^2$ satisfies
$\frac{3p+7}2\le |P|\le 2p+7$. If $P$ is not contained in a union of two
lines, then there is a direction in $\F_p^2$ such that some line in this
direction contains at least three points of $P$, and any line in this
direction contains at most $\frac{p+5}2$ points of $P$.
\end{lemma}

\begin{proof}
We say that a line $l\sbs\F_p^2$ is \emph{rich} if $|l\cap P|\ge 3$, and that
it is \emph{powerful} if $|l\cap P|\ge\frac{p+7}2$. Furthermore, we say that
a \emph{direction} in $\F_p^2$ is rich if there is a rich line in this
direction. There is at least one rich line: otherwise for any fixed point
$x\in P$, each of the $p+1$ lines through $x$ would contain at most one point
of $P$ other than $x$, leading to $|P|\le p+2$. Aiming at a contradiction, we
assume that there is a powerful line in every rich direction.

If we could find four distinct rich directions, then choosing a powerful line
in each of them and counting only those points of $P$ lying in the union of
these lines, we would get
  $$ |P| \ge 4\cdot \frac{p+7}2 - \binom42 = 2p + 8 , $$
a contradiction. This shows that there are at most three rich directions and,
consequently, at most three rich lines trough any point of $P$.

Let $l$ be a powerful line. If there is yet another powerful line, say $l'$,
which is parallel to $l$, then we fix a point $x\in P\stm(l\cup l')$ and with
every point $g\in l$ associate the point $g'\in l'$ such that $g,g'$ and $x$
are collinear. By the pigeonhole principle, there are at least
$2\cdot\frac{p+7}2-p=7$ pairs $(g,g')\in l\times l'$ such that both $g$ and
$g'$ belong to $P$; this shows that there are at least seven rich lines
through $x$, which, as we saw above, is impossible. A similar argument
applies if there is a powerful line $l'$ which is not parallel to $l$, except
that in this case the intersection point of $l$ and $l'$ gets associated to
itself, and there is a unique point on $l$ not associated to any point of
$l'$, and a unique point on $l'$ not associated to any point of $l$; this
results in at least $2\Big(\frac{p+7}2-2\Big)-(p-2)=5$ pairs
 $(g,g')\in l\times l'$ with $g\ne g'$ and $g,g'\in P$, and hence to at least
five rich lines through $x$, a contradiction.

We thus conclude that there is a unique powerful line $l$ and, consequently,
a unique rich direction. Fix a point $x\in P\stm l$. The line through $x$
parallel to $l$ is not powerful; therefore, contains at most $\frac{p+5}2$
points of $P$, including $x$ itself. Any other line through $x$ has the
direction other than that of $l$, and therefore is not rich; as a result,
contains at most one point of $P$ other than $x$. This shows that
  $$ |P| \le \frac{p+5}2 + p < \frac{3p+7}2, $$
a contradiction.
\end{proof}

\begin{lemma}\label{l:AQ}
Suppose that $p$ is a prime, $h\in L(\F_p)$ is a nonzero function, and
$A\seq\F_p$ is a set with $|A|>\frac23p$. If for any $a_1\longc a_4\in A$
such that $a_1+a_2=a_3+a_4$ we have $h(a_1)h(a_2)=h(a_3)h(a_4)$, then either
$|\supp\hh|=1$, or $|\supp\hh|\ge|A|$.
\end{lemma}

\begin{proof}
If $h$ vanishes on the whole set $A$, then $|\supp h|\le p-|A|$, whence
$|\supp\hh|\ge|A|+1$ by Theorem~\reft{birotao}. Suppose thus that the set
$A_1:=A\cap\supp h$ is nonempty, and let $A_0:=A\stm A_1$. If $A_0$ is
nonempty either, then $|A_0-A_1|\ge|A_0|+|A_1|-1=|A|-1$ by the well-known
Cauchy-Davenport theorem; consequently, $|A_0-A_1|+|A|\ge 2|A|-1>p$, and the
pigeonhole principle gives $A_0-A_1+A=\F_p$. Hence, for any $a\in A$ there
exist $a_0\in A_0$, $a_1\in A_1$, and $a'\in A$ with $a_0+a'=a_1+a$, implying
$h(a)=h(a_0)h(a')/h(a_1)=0$. This contradicts the assumption that $h$ does
not vanish on the whole set $A$, and thus shows that $A_0$ is empty; that is,
$A\seq\supp h$.

Since $|A|>\frac12\,p$, every element $g\in\F_p$ can be represented as
$g=a_1-a_2$ with $a_1,a_2\in A$, and we let $\chi(g):=h(a_1)/h(a_2)$; notice
that this definition is legitimate as for any other representation
$g=a_1'-a_2'$ with $a_1',a_2'\in A$ we have $h(a_1')/h(a_2')=h(a_1)/h(a_2)$.
We now claim that $\chi(g_1-g_2)=\chi(g_1)/\chi(g_2)$ for any
$g_1,g_2\in\F_p$. To see this, we notice that the intersection
$(A-g_1)\cap(A-g_2)\cap A$ is nonempty by the pigeonhole principle, and find
$a_1,a_2,a\in A$ with $a_1-g_1=a_2-g_2=a$; this gives
  $$ g_1=a_1-a,\ g_2=a_2-a,\ g_1-g_2=a_1-a_2, $$
as a result of which
  $$ \chi(g_1-g_2) = h(a_1)/h(a_2) = \chi(g_1) / \chi(g_2). $$
We conclude that $\chi$ is a character of the group $\F_p$. Moreover,
  $$ \chi(a_1-a_2) = h(a_1)/h(a_2),\ a_1,a_2\in A $$
shows that $h\ochi$ is constant on $A$; that is, there exists a nonzero
$C\in\C$ such that $h(a)=C\chi(a)$ for any $a\in A$. Consequently, the
difference function $\Del:=h-C\chi$ is supported outside of $A$. Hence,
either it is identically zero, or $|\supp\wh{\Del}|\ge p+1-(p-|A|)=|A|+1$ by
Theorem~\reft{birotao}, in which case $|\supp\hh|\ge|A|$.
\end{proof}

\begin{proof}[Proof of Theorem~\reft{as3}]
We assume, without loss of generality, that $|X|\le|S|$, and then for a
contradiction that
\begin{equation}\label{e:7/6SXsmall}
  |X| < 3(1-\eps)p \quad\text{and}\quad |S| < \frac16\,\eps p^{4/3},
\end{equation}
while $X$ is not contained in a union of two lines.

If we had $\eps^3\le 216p^{-1}$, then~\refe{7/6SXsmall} would give
$|X|\le|S|<p$, contradicting Theorem~\reft{basic}; hence $\eps^3>216p^{-1}$,
implying $p>216$.

If we had $|X|\le\frac{3p+5}2$, then Theorem~\reft{as2} would result in
  $$ |S| \ge \frac{\sqrt p}2\,\Big( 2p - \frac{3p+5}2 \Big)
                          = \frac14\,(p-5)\sqrt p> \frac16\,\eps p^{4/3}, $$
contradicting \refe{7/6SXsmall}. Thus,
\begin{equation}\label{e:7/6Xlarge}
  |X|\ge\frac{3p+7}2;
\end{equation}
combining this with~\refe{7/6SXsmall}, we get
  $$ \frac32\,p < |X|< 3(1-\eps)p, $$
and it follows that $\eps<1/2$. The inequalities
  $$ \frac{216}p < \eps^3 < \frac1{8}\quad  $$
are tacitly used in the computations below.

Recalling~\refe{7/6Xlarge} and applying Lemma~\refl{richlines} if $|X|\le
2p$, and Lemma~\refl{lines} if $2p<|X|<3p$, and observing that if $|X|>2p$,
then there is a line in every given direction containing at least three
points of $X$, we conclude that there is a nonzero, proper subgroup $H<G$
such that every $H^\perp$-line contains at most
  $$ \max \left\{ \frac{p+5}2,\sqrt{3p}+\frac32 \right\} = \frac{p+5}2 $$
points of $X$, while there is an $H^\perp$-line containing at least three
points of $X$. Throughout the proof, we consider this subgroup $H$ fixed, and
define $n_S,n_X,K_S,K_X$ by~\refe{nSnX} and~\refe{KSKX}.

From~\refe{7/6SXsmall} and the assumption that $X$ is not contained in a
union of two lines, we get
\begin{equation*}\label{e:nXsqrt73}
  n_X \le \frac13\,|X| < \,(1-\eps)p,
\end{equation*}
whence, by~\refe{7/6SXsmall} and Lemma~\refl{SXmn},
  $$ \frac16\,\eps p^{4/3} > |S| \ge n_S(p+1-n_X) > \eps n_Sp; $$
consequently,
  $$ n_S < \frac16\,p^{1/3}. $$
Applying~\refe{7/6SXsmall} and Lemma~\refl{SXmn} once again,
  $$ 3\,(1-\eps)p > |X| \ge n_X(p+1-n_S)
                    > n_X \Big(p-\frac16\,p^{1/3} \Big) > (1-\eps)n_Xp. $$
This gives $n_X\le 2$. As a result, by Lemma~\refl{SXmn}, we have
 $K_S\ge p-1$. Hence, averaging and using~\refe{7/6SXsmall},
\begin{equation}\label{e:newmsqrt}
  n_S \le K_S^{-1}|S| < \frac{\eps}6\, \frac{ p^{4/3}}{p-1}.
\end{equation}

Denoting by $N$ the number of $H^\perp$-lines containing exactly one or
exactly two points of $X$, we have by Lemma~\refl{SXmn}
  $$ |X| \ge N + 3(K_X-N) \ge N + 3(p +1- n_S - N), $$
whence
\begin{equation}\label{e:lowerboundforN}
  2N > 3(p - n_S)- |X|.
\end{equation}

Fix an element $\gam\in G\stm H$, and consider the functions
  $$ F_g := f\cdot(1_{g+\gam+H}-1_{g+H}),\quad g\in G. $$
By~\refe{psf}, we have
\begin{equation}\label{e:Fghatt}
  \wh{F_g}(\chi) = \frac1p\,\sum_{\psi\in H^\perp}
                         \hf(\chi\psi)(\psi(\gam)-1) \psi(g),\ \chi\in\hG.
\end{equation}
In the case where the $H^\perp$-line through a character $\chi\in X$ contains
exactly one more character of $X$, say $\chi\psi$ with some $\psi\in
H^\perp$, this gives
  $$ \wh{F_g}(\chi) = \frac1p\,\hf(\chi\psi)(\psi(\gam)-1)\psi(g); $$
it follows that if $g_1,g_2,g_3,g_4\in G$ satisfy
\begin{equation}\label{e:new7/6gi}
  g_1+g_2=g_3+g_4,
\end{equation}
then
  $$ \wh{F_{g_1}}(\chi)\wh{F_{g_2}}(\chi)
                               - \wh{F_{g_3}}(\chi)\wh{F_{g_4}}(\chi) = 0. $$
This conclusion stays true also if the $H^\perp$-line through $\chi$ does not
contain any points of $X$ other than $\chi$, as in this case,
by~\refe{Fghatt}, we have $\wh{F_{g_i}}(\chi)=0$ for each $i\in[1,4]$.

For a quadruple $\bg=(g_1,g_2,g_3,g_4)\in G^4$ satisfying~\refe{new7/6gi},
let
  $$ \Del_\bg := f\ast(F_{g_1}\ast F_{g_2}-F_{g_3}\ast F_{g_4}); $$
thus,
\begin{equation}\label{e:Delbghatt}
  \wh{\Del_\bg}(\chi) = \hf(\chi)\,\big(\wh{F_{g_1}}(\chi)\wh{F_{g_2}}(\chi)
                - \wh{F_{g_3}}(\chi)\wh{F_{g_4}}(\chi)\big),\quad \chi\in\hG.
\end{equation}
Write
\begin{gather*}
  T := \supp\Del_\bg,\ Y := \supp\wh{\Del_\bg}, \\
  n_T := \min\{|(t+H)\cap T|\colon t\in T\},\ %
                    n_Y:=\min\{|(\chi H^\perp)\cap Y|\colon \chi\in Y \}, \\
  K_T := |T+H|/|H|,\ K_Y := |YH^\perp|/|H^\perp|.
\end{gather*}
By~\refe{Delbghatt}, we have $Y\seq X$, and we have shown above that
$\wh{\Del_\bg}(\chi)=0$ for any character $\chi\in X$ with
 $|\chi H^\perp\cap X|\le 2$. Along with~\refe{7/6SXsmall},
\refe{lowerboundforN}, and~\refe{newmsqrt}, this gives
  $$ 2K_Y \le 2p - 2N < |X| - p + 3n_S < 2\Big(1-\frac54\,\eps\Big)p, $$
implying
\begin{equation}\label{e:nTlarge}
  n_T \ge p + 1 - K_Y > \frac54\,\eps p
\end{equation}
in view of Lemma~\refl{SXmn}. On the other hand, letting
  $$ k_g = |(S-g)\cap H| $$
(so that $k_g=|\supp(f\cdot1_{g+H})|$) we have
\begin{equation}\label{e:new76TSl}
  |T| \le |S|\big( (k_{g_1}+k_{g_1+\gam})(k_{g_2}+k_{g_2+\gam})
                    +(k_{g_3}+k_{g_3+\gam})(k_{g_4}+k_{g_4+\gam}) \big).
\end{equation}

Suppose that $\Del_\bg\ne 0$. By the assumption that every $H^\perp$-line
contains at most $\frac{p+5}2$ points of $X$, and since $Y\seq X$, we have
  $$ n_Y \le \frac{p+5}{2} $$
and then, by Lemma~\refl{SXmn} and~\refe{nTlarge},
  $$ |T| \ge n_T(p+1-n_Y) > \frac54\,\eps p \cdot \frac{p-3}{2}
                                                      > \frac12\,\eps p^2. $$
Comparing this with~\refe{new76TSl}, we obtain
\begin{equation}\label{e:new76Slg}
  |S| \big( (k_{g_1}+k_{g_1+\gam})(k_{g_2}+k_{g_2+\gam})
             +(k_{g_3}+k_{g_3+\gam})(k_{g_4}+k_{g_4+\gam}) \big)
                                                > \frac12\,\eps p^{2},
\end{equation}
provided that $\bg=(g_1\longc g_4)\in G^4$ satisfies~\refe{new7/6gi}, and
$\Del_\bg\ne 0$.

Let $\Gam<G$ be the subgroup generated by $\gam$. We have
  $$ \sum_{g\in\Gam} (k_g+k_{g+\gam})
                                 = 2 \sum_{g\in\Gam} |S\cap(g+H)| = 2|S|; $$
as a result, denoting by $A$ the set of all those $g\in\Gam$ with
$k_g+k_{g+\gam}<6|S|/p$, we have $|A|>\frac23\,p$. Moreover, as it follows
from~\refe{new76Slg} and~\refe{7/6SXsmall}, if $\bg=(g_1\longc g_4)\in A^4$
satisfies~\refe{new7/6gi}, and $\Del_\bg\ne 0$, then
  $$ \frac12\,\eps p^{2} < 72|S|^3p^{-2}
                                     < 72 \cdot\frac1{216}\,\eps^3 p^{2}, $$
a contradiction.

We conclude that for any $\bg\in A^4$ satisfying~\refe{new7/6gi}, we have
$\Del_\bg=0$; that is, by~\refe{Delbghatt},
  $$ \wh{F_{g_1}}(\chi)\wh{F_{g_2}}(\chi)
                  = \wh{F_{g_3}}(\chi)\wh{F_{g_4}}(\chi),\quad \chi\in X. $$
For every character $\chi\in X$ we now consider the function $h_\chi\in
L(\Gam)$ defined by
  $$ h_\chi(g) := \wh{F_g}(\chi),\ g\in\Gam. $$
Identifying $H^\perp$ with the character group $\wh{\Gam}$, in view
of~\refe{Fghatt} we have $|\supp\wh{h_\chi}|=\kappa_\chi-1$, where
$\kappa_\chi=|\chi H^\perp\cap X|$ is the number of points of $X$ on the
$H^\perp$-line through $\chi$. Applying Lemma~\refl{AQ} we derive that either
$\kappa_\chi\le 2$, or $\kappa_\chi\ge|A|+1>\frac23\,p>\frac{p+5}2$, for any
character $\chi\in X$. This, however, contradicts the choice of $H$ at the
beginning of the proof.
\end{proof}

\appendix
\section*{Appendix: Classifying the exceptions}\label{s:exceptions}

In this section we prove Lemmas~\refl{cosets1}--\refl{twononparallel}
classifying the exceptional cases of Theorems~\reft{rational}--\reft{as3},
and also prove Corollary~\refc{uppergray}.

Recall, that for a subgroup $H$ of a finite abelian group $G$, a function
$f\in L(G)$ is called \emph{$H$-periodic} if $f(g+h)=f(g)$ for any $g\in G$
and $h\in H$. A set $S\seq G$ is $H$-periodic if its indicator function $1_S$
is $H$-periodic; that is, $g\in S$ if and only if $g+h\in S$, for any $g\in
G$ and $h\in H$. Equivalently, a function $f$ is $H$-periodic if it is
constant on $H$-cosets, and a set $S$ is $H$-periodic if it is a union of
$H$-cosets.

\begin{lemma}\label{l:cosets1}
Suppose that $H$ is a subgroup of a finite abelian group $G$, and
 $f\in L(G)$.
\begin{itemize}
\item[i)] We have $\supp f\seq H$ if and only if $\hf$ is
    $H^\perp$-periodic. Also, if $\supp f\seq g+H$ for some $g\in G$,
    then $\supp\hf$ is $H^\perp$-periodic.
\item[ii)] We have $\supp\hf\seq H^\perp$ if and only if $f$ is
    $H$-periodic. Also, if $\supp\hf\seq\chi H^\perp$ for some
    $\chi\in\hG$, then $\supp f$ is $H$-periodic.
\end{itemize}
\end{lemma}
We omit the straightforward verification.

\begin{lemma}\label{l:cosets2}
Suppose that $H$ is a subgroup, $g$ is an element, and $\chi\in\hG$ is a
character of a finite abelian group $G$, and that $f\in L(G)$ is a nonzero
function.
\begin{itemize}
\item[i)] If $\supp f\seq g+H$ and $\supp\hf=\chi_1 H^\perp\longu\chi_k
    H^\perp$ where $\chi_1\longc\chi_k\in\hG$ and the union is disjoint
    (cf.~Lemma~\refl{cosets1}), then there are nonzero coefficients
    $c_1\longc c_k\in\C$ such that
    $$ f(z) = \begin{cases}
                c_1\chi_1(z)\longp c_k\chi_k(z)\ &\text{if}\ z\in g+H, \\
                0                              \ &\text{if}\ z\notin g+H.
            \end{cases} $$
\item[ii)] If $\supp\hf\seq\chi H^\perp$ and $\supp
    f=(g_1+H)\longu(g_k+H)$ where $g_1\longc g_k\in G$ and the union is
    disjoint (cf.~Lemma~\refl{cosets1}), then there are nonzero
    coefficients $c_1\longc c_k\in\C$ such that
    $$ f(z) = \begin{cases}
                c_i\chi(z)\ &\text{if}\ z\in g_i+H,\ i\in[1,k], \\
                0   \ &\text{if}\ z\notin(g_1+H)\longu(g_k+H).
            \end{cases} $$
\end{itemize}
\end{lemma}

\begin{proof}
For the first part of the lemma we notice that for any $z\in g+H$, by the
inversion formula we have
  $$ f(z) = \sum_{i=1}^k \sum_{\psi\in H^\perp}
              \hf(\chi_i\psi)\chi_i(z)\psi(z) = \sum_{i=1}^k c_i\chi_i(z) $$
where, by~\refe{psf} and in view of $f\cdot1_{g+H}=f$,
  $$ c_i = \sum_{\psi\in H^\perp} \hf(\chi_i\psi) \psi(g)
       = |H^\perp| \wh{f\cdot1_{g+H}}(\chi_i)
       = |H^\perp| \hf(\chi_i) \ne 0. $$
This proves the first assertion.

Turning to the second assertion, we observe that the inversion formula gives
  $$ f(z) = \sum_{\psi\in H^\perp} \hf(\chi\psi) \chi(z)\psi(z),\ z\in G. $$
It follows that the function $\ochi f$ is $H$-periodic, and since $\supp\ochi
f=\supp f$, this function is constant and nonzero on each coset $g_i+H,\
i\in[1,k]$. Denoting by $c_i$ its values on the corresponding cosets
completes the proof.
\end{proof}

\begin{lemma}\label{l:twoparallel}
Suppose that $G$ is a finite abelian group, $H<G$ is a proper, prime
subgroup, and $\chi_1,\chi_2\in\hG$ are characters with
 $\chi_2 H^\perp\ne\chi_1 H^\perp$. For a function $f\in L(G)$, we
have $\supp\hf\seq\chi_1 H^\perp\cup\chi_2 H^\perp$ if and only if there are
$H$-periodic functions $f_1,f_2\in L(G)$ such that
  $$ f(g) = \chi_1(g)f_1(g)+\chi_2(g)f_2(g),\ g\in G. $$
Moreover, writing in this case $N:=|\supp f_1\cup\supp f_2|$, we have
  $$ \bigg(1-\frac1{|H|}\bigg)\,N \le |\supp f| \le N. $$
\end{lemma}

\begin{proof}
By the inversion formula, if $\supp\hf\seq\chi_1 H^\perp\cup\chi_2 H^\perp$,
then for any $g\in G$ we have
  $$ f(g) = \chi_1(g) \sum_{\psi\in H^\perp}\hf(\chi_1\psi)\psi(g)
           + \chi_2(g) \sum_{\psi\in H^\perp}\hf(\chi_2\psi)\psi(g), $$
and the existence of the functions $f_1,f_2$ follows by observing that the
two sums in the right-hand side depend on the coset $g+H$ only. Conversely,
it is easily seen that if $f$ is of the indicated form, then
 $\supp\hf\seq\chi_1 H^\perp\cup\chi_2 H^\perp$.

Furthermore, the estimate $|\supp f|\le N$ is immediate. For the remaining
estimate $|\supp f|\ge \big(1-|H|^{-1}\big)\,N$, we notice that if, for some
$g\in G$, at least one of $f_1(g)$ and $f_2(g)$ is nonzero, then all but at
most one element of the coset $g+H$ lie in $\supp f$, as it follows from the
nonsingularity of the matrices
  $$ \begin{pmatrix}
       \chi_1(g_1) & \chi_2(g_1) \\
       \chi_1(g_2) & \chi_2(g_2)
     \end{pmatrix}, \qquad g_1,g_2\in g+H,\ g_1\ne g_2 $$
(this is where primality of $H$ is required).
\end{proof}

\begin{lemma}\label{l:twononparallel}
Suppose that $G=H_1\oplus H_2$ is a decomposition of the finite abelian group
$G$ into a direct sum of nonzero subgroups $H_1,H_2<G$, and let
 $f\in L(G)$. For $\supp\hf$ to be contained in a union of a coset of
$H_1^\perp$ and a coset of $H_2^\perp$, it is necessary and sufficient that
there existed functions $f_1\in L(H_1)$ and $f_2\in L(H_2)$ and a character
$\chi\in\hG$ such that
  $$ f(h_1+h_2) = \chi(h_1+h_2)\big(f_1(h_1)+f_2(h_2)\big),
                                              \quad h_1\in H_1,h_2\in H_2. $$
Moreover, if in this case $|\supp f|<\frac12\,|G|$, then the functions
$f_1,f_2$ can be so chosen that
  $$ |\supp f| \le |H_1||\supp f_2|+|H_2||\supp f_1|
                      \le \bigg(1+\frac{2|\supp f|}{|G|}\bigg)\,|\supp f|. $$
\end{lemma}

\begin{proof}
Sufficiency is easy to verify. For the necessity, let $\chi$ be the character
lying in both cosets on which $\hf$ is supported. By the inversion formula,
for any $h_1\in H_1$ and $h_2\in H_2$ we have
\begin{align*}
  f(h_1+h_2) &= \sum_{\psi\in H_1^\perp} \hf(\chi\psi)\,\chi\psi(h_1+h_2)
                   + \sum_{\sub{\psi\in H_2^\perp \\ \psi\ne 1}}
                                          \hf(\chi\psi)\,\chi\psi(h_1+h_2) \\
       &= \chi(h_1+h_2) \sum_{\psi\in H_1^\perp} \hf(\chi\psi)\psi(h_2)
            + \chi(h_1+h_2) \sum_{\sub{\psi\in H_2^\perp \\ \psi\ne 1}}
                                                    \hf(\chi\psi)\psi(h_1),
\end{align*}
and we let
  $$ f_1(h_1) := \sum_{\sub{\psi\in H_2^\perp \\ \psi\ne 1}}
                                      \hf(\chi\psi)\psi(h_1),\ h_1\in H_1 $$
and
  $$ f_2(h_2)
         := \sum_{\psi\in H_1^\perp} \hf(\chi\psi)\psi(h_2),\ h_2\in H_2. $$

Turning to the second assertion, the inequality
  $$ |\supp f| \le |H_2||\supp f_1| + |H_1||\supp f_2| $$
is immediate (if $h_1+h_2\in\supp f$, where $h_1\in H_1,\ h_2\in H_2$, then
either $h_1\in\supp f_1$, or $h_2\in\supp f_2$), and we proceed to prove the
remaining inequality. We write for brevity $S:=\supp f$ and $Z:=G\stm S$, and
assume that $|Z|>\frac12\,|G|$. For $j\in\{1,2\}$ let
  $$ I_j:=\mathrm{Im}(f_j)\ \text{and}
                      \ \nu_j(z):=|\{h\in H_j\colon f_j(h)=z\},\ z\in\C. $$
Having $f_1$ and $f_2$ suitably translated, we further assume that
  $$ \nu_1(0) = \max \{ \nu_1(z)\colon z\in I_1 \}, $$
and we choose $z_0\in I_2$ with
  $$ \nu_2(z_0) = \max \{ \nu_2(z)\colon z\in I_2 \} $$
and write $m_1:=\nu_1(0)$ and $m_2:=\nu_2(z_0)$.

From
  $$ \frac12\,|G| < |Z| = \sum_{z\in I_1\cap(-I_2)} \nu_1(z)\nu_2(-z)
            \le m_1 \sum_{z\in I_2}\nu_2(z) = m_1|H_2| $$
we conclude that $m_1>\frac12\,|H_1|$, and similarly $m_2>\frac12\,|H_2|$.
Consequently, if we had $z_0\ne 0$, this would imply
\begin{align*}
  \frac12\,|G|
     &<   |Z| \\
     &\le m_1(|H_2|-m_2)+(|H_1|-m_1)m_2 \\
     &=   m_1|H_2| - (2m_1-|H_1|) m_2 \\
     &<   m_1|H_2| - (2m_1-|H_1|)\cdot\frac12\,|H_2| \\
     &=   \frac12\,|G|,
\end{align*}
a contradiction. Thus, $z_0=0$; as a result, writing $n_j:=|H_j|-m_j$ and
observing that $n_j<\frac12\,|H_j|\ (j\in\{1,2\})$ we get
\begin{align*}
  |S| &\ge m_1(|H_2|-m_2)+(|H_1|-m_1)m_2 \\
      &=   n_1|H_2|+n_2|H_1|-2n_1n_2 \\
      &\ge \max\{n_1|H_2|,n_2|H_1|\}.
\end{align*}
Hence, $n_1n_2\le |S|^2/|G|$, which further leads to
  $$ n_1|H_2|+n_2|H_1| \le |S| + 2n_1n_2 \le |S| + \frac{2|S|^2}{|G|}. $$
It remains to notice that $n_j=|\supp f_j|,\ j\in\{1,2\}$.
\end{proof}

Finally, we prove Corollary~\refc{uppergray}.
\begin{proof}[Proof of Corollary~\refc{uppergray}]
Without loss of generality we assume that $3\le|X|\le|S|$, and that we are
not in the exceptional situation where $X$ is a coset of a nonzero subgroup
of the corresponding group, possibly with one element missing, and $S$ is
either a coset, or a union of two cosets of the orthogonal subgroup. Also, we
assume that $|S|\le p^2-2p$ as otherwise the assertion follows in view of
$|X|\ge 3$.

If $|S|\le 2p-1$, then we use Theorem~\reft{k=p-1} to get
$\frac1{p-1}\,|X|+\frac12\,|S|\ge p+1$; this gives
  $$ |S||X| \ge \frac12\,(p-1)|S|(2p+2-|S|)
                            \ge \frac12\,(p-1)\cdot3(2p-1) > 3p(p-2). $$

If $2p\le|S|\le 3p-1$, then we apply Theorem~\reft{k=p-2} to get either
$\frac1{p-2}\,|X|+\frac13\,|S|\ge p+1$, or $|X|\ge\frac32(p-1)$. In the
former case
  $$ |S||X| \ge \frac13(p-2)|S|(3p+3-|S|)
                                 \ge \frac43\,(p-2)(3p-1) > 3p(p-2), $$
in the latter case $|S||X|\ge 3p(p-1)>3p(p-2)$.

Finally, if $|S|\ge 3p$, then using Theorem~\reft{meshulam} and the
assumption $|S|\le p^2-2p$ we get
  $$ |S||X| \ge \frac1p\,|S|\left(p^2+p-|S|\right) \ge 3p(p-2). $$
\end{proof}

\section*{Acknowledgement}
We are grateful to Miki Simonovits for his kind assistance with the
illustrations.

\bigskip

\end{document}